\documentclass[12pt,reqno]{amsart} 
\usepackage{amsfonts,amsmath,amsthm,amssymb}
\usepackage{fullpage}
\usepackage{hyperref}
\usepackage{tikz}

\usepackage{mathtools}

\newtheorem{theorem}{Theorem}

\newtheorem{corollary}[theorem]{Corollary}
\newtheorem{remark}[theorem]{Remark}

\newcommand{\PF}{\textrm{PF}}
\newcommand{\Lucky}{\text{Lucky}}

\definecolor{black}{RGB}{0,0,0}
\definecolor{orange}{RGB}{230,159,0}
\definecolor{skyblue}{RGB}{86,180,233}
\definecolor{bluishgreen}{RGB}{0,158,115}
\definecolor{yellow}{RGB}{240,228,66}
\definecolor{blue}{RGB}{0,114,178}
\definecolor{vermillion}{RGB}{213,94,0}
\definecolor{reddishpurple}{RGB}
{204,121,167}
\definecolor{grayish}{RGB}{211, 211, 211}
\definecolor{darkgray}{RGB}{128, 128, 128}

\begin{document}

\subjclass[2020]{Primary 05A05; Secondary 68P10}

\keywords{Parking Functions, Lucky Cars, Sorting}

\title{Lucky Cars and the \texttt{Quicksort} Algorithm}

\author[Harris]{Pamela E. Harris}
\address[P.~E. Harris]{Department of Mathematical Sciences, University of Wisconsin-Milwaukee, Milwaukee, WI 53211}
\email{\textcolor{blue}{\href{mailto:peharris@uwm.edu}{peharris@uwm.edu}}}
\thanks{P.~E.~Harris was supported through a Karen Uhlenbeck EDGE Fellowship.}

\author{Jan Kretschmann}
\address[J. Kretschmann]{Department of Mathematical Sciences, University of Wisconsin Milwaukee, Milwaukee, WI 53211}
\email{\textcolor{blue}{\href{mailto:kretsc23@uwm.edu}{kretsc23@uwm.edu}}}

\author[Mart\'inez Mori]{J. Carlos Mart\'{i}nez Mori}
\address[J.C. Mart\'inez Mori]{Schmidt Science Fellows}
\email{\textcolor{blue}{\href{mailto:jmartinezmori@schmidtsciencefellows.org}{jmartinezmori@schmidtsciencefellows.org}}}
	
\maketitle

\begin{abstract}
\texttt{Quicksort} is a classical divide-and-conquer sorting algorithm.
It is a comparison sort that makes an average of $2(n+1)H_n - 4n$ comparisons on an array of size $n$ ordered uniformly at random, where $H_n \coloneqq \sum_{i=1}^n\frac{1}{i}$ is the $n$th harmonic number.
Therefore, it makes $n!\left[2(n+1)H_n - 4n\right]$ comparisons to sort all possible orderings of the array.
In this article, we prove that this count also enumerates 
the parking preference lists of $n$ cars parking on a one-way street with $n$ parking spots resulting in exactly $n-1$ lucky cars (i.e., cars that park in their preferred spot). 
For $n\geq 2$, both counts satisfy the second order recurrence relation
$    f_n=2nf_{n-1}-n(n-1)f_{n-2}+2(n-1)!
$ with $f_0=f_1=0$.
\end{abstract}

\section{Introduction}
A common student experience is to form a line where people are sorted by their height, say with the shortest student on the left and the tallest student on the right.
One approach to sort the students is to have students join the line one at a time, each finding their place in line by \emph{comparing} their height to that of those already in it. 
Thereby one at a time, students self-sort and the end result yields the desired student ordering.

As inconspicuous as this activity may seem, sorting algorithms and their use of computational resources (e.g., comparisons) are subject to vast amounts of scientific attention.
One such algorithm, aptly called \texttt{Quicksort}, was developed by was developed by Tony Hoare---who would later become Sir Tony Hoare--- in 1961 while he was a British Council exchange student visiting Moscow State University~\cite{foley1971proofCormen,hoare1961algorithm,jones2021theories}.
\texttt{Quicksort} is a classical sorting algorithm and a prime example of the \emph{divide-and-conquer} paradigm. 
By now, its design and analysis is covered in most undergraduate textbooks on algorithm design (e.g.,~\cite{cormen2009introduction,kleinberg2006algorithm,knuth1998art, sedgewick2014algorithms}, to list a few).

The algorithm first bipartitions the array elements by selecting a \emph{pivot} element:
elements less than or equal to the pivot element are assigned to a ``left'' array, whereas elements greater than the pivot element are assigned to a ``right'' array. 
In this way, given an array with $n \in \mathbb{N} \coloneqq \{1,2,3,\ldots\}$ elements, the algorithm first makes $n - 1$ pivot comparisons.
Then, the algorithm recurses \emph{separately} on the left and right arrays to obtain their respective sorted versions\footnote{
In practice, \texttt{Quicksort} is implemented in conjunction with non-recursive sorting algorithms that work well on small arrays.
This and other practical optimizations of the algorithm are described by Sedgewick in~\cite{sedgewick1978implementing}.
}.

In Figure \ref{fig:quicksort-ex}, we illustrate an execution of the algorithm as described in \cite[Chapter 13.5]{kleinberg2006algorithm}.
(A more detailed description, with precise array-indexing and swap operations, can be found in \cite[Chapter 7]{cormen2009introduction}).

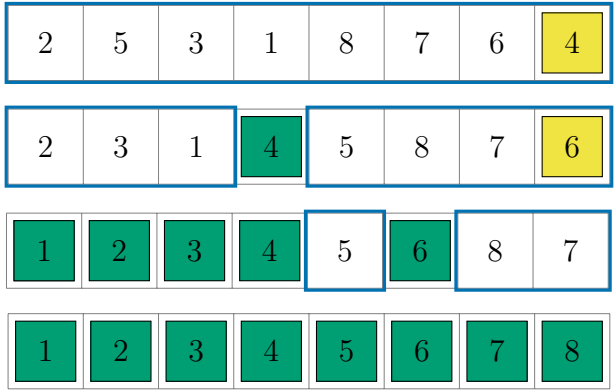
\begin{figure}
    \centering
    \begin{tikzpicture}
        \draw[step=1cm,gray,very thin] (0,0) grid (8,1);
        \draw[blue,very thick] (-.025,-.025) rectangle (8.025,1.025);
        \node at (.5,.5) {$2$};
        \node at (1.5,.5) {$5$};
        \node at (2.5,.5) {$3$};
        \node at (3.5,.5) {$1$};
        \node at (4.5,.5) {$8$};
        \node at (5.5,.5) {$7$};
        \node at (6.5,.5) {$6$};
        \draw[fill=yellow] (7.1,0.1) rectangle (7.9,0.9);
        \node at (7.5,.5) {\textcolor{black}{{$4$}}};
    \end{tikzpicture} 
\\
\vspace{.1in}

    \begin{tikzpicture}
        \draw[step=1cm,gray,very thin] (0,0) grid (8,1);
        \draw[blue,very thick] (-.025,-.025) rectangle (3.025,1.025);
        \draw[blue,very thick] (3.975,-.025) rectangle (8.025,1.025);
        
        \node at (.5,.5) {$2$};
        \node at (1.5,.5) {$3$};
        \node at (2.5,.5) {$1$};
        \draw[fill=bluishgreen] (3.1,0.1) rectangle (3.9,0.9);
        \node at (3.5,.5) {\textbf{$4$}};        
        \node at (4.5,.5) {$5$};
        \node at (5.5,.5) {$8$};
        \node at (6.5,.5) {$7$};
        \draw[fill=yellow] (7.1,0.1) rectangle (7.9,0.9);
        \node at (7.5,.5) {\textcolor{black}{${6}$}};
    \end{tikzpicture}
\\
\vspace{.1in}
    \begin{tikzpicture}
        \draw[step=1cm,gray,very thin] (0,0) grid (8,1);
        \draw[fill=bluishgreen] (.1,0.1) rectangle (.9,0.9);
        \draw[fill=bluishgreen] (.1,0.1) rectangle (.9,0.9);
        \node at (.5,.5) {$1$};
        \draw[fill=bluishgreen] (1.1,0.1) rectangle (1.9,0.9);
        \node at (1.5,.5) {$2$};
        \draw[fill=bluishgreen] (2.1,0.1) rectangle (2.9,0.9);
        \node at (2.5,.5) {$3$};
        \draw[fill=bluishgreen] (3.1,0.1) rectangle (3.9,0.9);
        \node at (3.5,.5) {\textbf{$4$}};        
        
        \draw[blue,very thick] (3.975,-.025) rectangle (5.025,1.025);
        \draw[blue,very thick] (5.975,-.025) rectangle (8.025,1.025);
        
        \node at (4.5,.5) {$5$};
        \draw[fill=bluishgreen] (5.1,0.1) rectangle (5.9,0.9);
        \node at (5.5,.5) {$6$};
        \node at (6.5,.5) {$8$};
        \node at (7.5,.5) {$7$};
    \end{tikzpicture}
    \\
\vspace{.1in}
    \begin{tikzpicture}
        \draw[step=1cm,gray,very thin] (0,0) grid (8,1);
        \draw[fill=bluishgreen] (.1,0.1) rectangle (.9,0.9);
        \node at (.5,.5) {$1$};
        \draw[fill=bluishgreen] (1.1,0.1) rectangle (1.9,0.9);
        \node at (1.5,.5) {$2$};
        \draw[fill=bluishgreen] (2.1,0.1) rectangle (2.9,0.9);
        \node at (2.5,.5) {$3$};
        \draw[fill=bluishgreen] (3.1,0.1) rectangle (3.9,0.9);
        \node at (3.5,.5) {$4$};
        \draw[fill=bluishgreen] (4.1,0.1) rectangle (4.9,0.9);
        \node at (4.5,.5) {$5$};
        \draw[fill=bluishgreen] (5.1,0.1) rectangle (5.9,0.9);
        \node at (5.5,.5) {$6$};
        \draw[fill=bluishgreen] (6.1,0.1) rectangle (6.9,0.9);
        \node at (6.5,.5) {$7$};
        \draw[fill=bluishgreen] (7.1,0.1) rectangle (7.9,0.9);
        \node at (7.5,.5) {$8$};
    \end{tikzpicture}

    \caption{
    Using Quicksort to sort the permutation 25318764. The rightmost subarray element is chosen as the pivot, highlighted in yellow. 
    The algorithm partitions the array into elements larger and smaller than the pivot. 
    If the subarray is length 3 or smaller, it sorts by brute force. 
    Otherwise, it recurses on the resulting subarrays, marked in blue.
    Correctly ordered elements are highlighted in green.
    }
    \label{fig:quicksort-ex}
\end{figure}

Given an array of size $n$, the \emph{overall} number of pivot comparisons made by \texttt{Quicksort} (including all recursive calls) crucially depends on the choice of pivot elements.
For example, if the algorithm happens to repeatedly select the smallest element as the pivot element, then the input array of any particular recursive call is only one unit smaller than that of the preceding call. 
It follows that, in the worst case, the algorithm makes $\sum_{i=1}^{n} (n - i) = \frac{n(n-1)}{2}= O(n^2)$ pivot comparisons\footnote{Given $f, g: \mathbb{N} \rightarrow \mathbb{N} $, we say $f = O(g)$ is there exist $k \in \mathbb{N}$ and $C \in \mathbb{R}_{> 0}$ such that $f(n) \leq C \cdot g(n)$ for all $n \geq k$.}.
However, if we assume the original array is ordered uniformly at random (i.e., it is in any given ordering with probability $1/n!$), the expected number of pivot comparisons is $O(n \log n)$.
This bound can be obtained as follows:
Let $Q_n$ be the expected number of pivot comparisons performed by \texttt{Quicksort} given an array of size $n$ ordered uniformly at random.
Then, regardless of the pivot selection strategy (e.g., selecting the right-most element), \emph{any} choice of pivot element is equally likely to be the $k$th smallest and we have
\begin{equation}
\label{eq: Q_n recursive}
    Q_n = (n-1) + \frac{1}{n}\left(\sum_{k=1}^{n} Q_{k-1} + Q_{n - k} \right).
\end{equation}
This recursive relation unravels (see~\cite[Chapter~4.7]{cameron1994combinatorics} for details) to obtain
\begin{equation}
\label{eq: Q_n formula}
    Q_n = 2(n + 1) H_n - 4n = O(n \log n),
\end{equation}
where $H_n$ is the $n$th harmonic number, i.e.~$H_n \coloneqq \sum_{i=1}^n \frac1i$.

Note that while $Q_n$ is not an integer sequence, $A_n=n!Q_n$ is an integer sequence and denotes the \emph{total} number of comparisons performed by \texttt{Quicksort} given \emph{all} possible orderings of an array of size $n$, (see \cite[\href{https://oeis.org/A288964}{A288964}]{OEIS}).

In fact, by the definition of expectation, $A_n$ is given by
\begin{equation}
\label{eq: A_n formula}
    A_n = n! Q_n=n!\left[2(n + 1) H_n - 4n\right].
\end{equation}
\section{Main Result}
In this work, we give a new and surprising connection between the sequence $A_n$ and a certain class of ``parking objects.''
To describe this connection, we begin with some preliminaries on parking functions.
For $n \in \mathbb{N}$ we let $[n] \coloneqq \{1, 2, \ldots, n\}$.
Suppose there are $n$ cars, numbered $1, 2, \ldots, n$, which enter (in order) a one-way street with $n$ parking spots.
Each car $i \in [n]$ has a preferred parking spot $p_i \in [n]$.
Car $i$ parks in the first unoccupied spot numbered greater than or equal to its preference, if such a spot exists.
If no such spot exists, car $i$ is unable to park.
We say $\alpha=(p_1,p_2,\ldots,p_n)$ is a \emph{preference list} of length $n$.
If $\alpha$ allows \emph{all} cars to park, we say it is a \emph{parking function} of length $n$. 
For example, as illustrated in Figure \ref{fig:pf example}, $\alpha=(3,1,1,2)$ is a parking function of length 4, where car 1 parks in spot 3, car 2 parks in spot 1, car 3 parks in spot 2, and car 4 parks in spot 4, while $\alpha=(2,1,4,4)$ is not a parking function as car 4 is unable to park.
Throughout we let $[n]^n$ denote the set of preference lists of length $n$ and $\PF_n \subseteq [n]^n$ denote the set of parking functions of length $n$.

Parking functions were introduced by Konheim and Weiss~\cite{Konheim_Weiss} in their 1966 study of hashing functions.
They showed that $|\PF_n|=(n+1)^{n-1}$.
Since then, there has been much interest in parking functions and related combinatorial objects, which have been shown to have connections to a wide variety of areas such as hyperplane arrangements~\cite{Hyperplanes}, polyhedral combinatorics~\cite{Polytopes, Flow}, and the Tower of Hanoi game~\cite{Tower}, to list a few. 
Many have studied variants of parking functions to include cars that move backwards when they find their preferred space occupied~\cite{Naples, NaplesCount}, cars having a set of preferences rather than a single preference~\cite{Interval, Subset}, cars bumping earlier cars from their preferred spot~\cite{MVP}, cars having a variety of lengths~\cite{Happ,Adeniran,chen2022permutation,franks2023counting}, and cars arriving after construction on the street made certain spots unavailable for parking \cite{Completions}.
For a technical survey on parking functions we recommend~\cite{YanHEC}, and for those interested in open problems related to parking functions we recommend~\cite{Choose}.

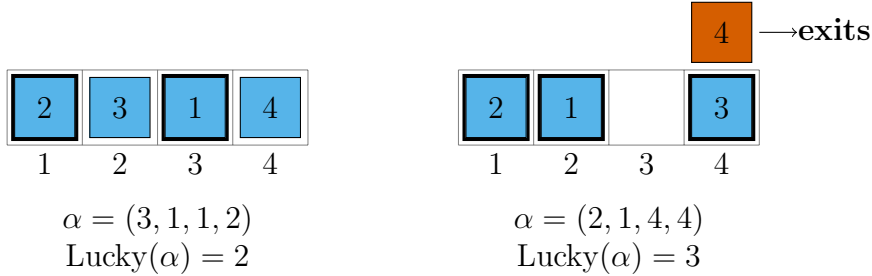
\begin{figure}
    \centering
    \begin{tikzpicture}
        \draw[black,very thick] (0.075,0.075) rectangle (0.925,0.925);
        \draw[black,very thick] (2.075,0.075) rectangle (2.925,0.925);
        \draw[step=1cm,gray,very thin] (0,0) grid (4,1);
        \draw[fill=skyblue] (.1,0.1) rectangle (.9,0.9);
        \node at (.5,.5) {\textbf{$2$}};
        \node at (.5,-.25) {\textbf{$1$}};
        \draw[fill=skyblue] (1.1,0.1) rectangle (1.9,0.9);
        \node at (1.5,.5) {\textbf{$3$}};        
        \node at (1.5,-.25) {\textbf{$2$}};        
        \draw[fill=skyblue] (2.1,0.1) rectangle (2.9,0.9);
        \node at (2.5,.5) {\textbf{$1$}};        
        \node at (2.5,-.25) {\textbf{$3$}};        
        \draw[fill=skyblue] (3.1,0.1) rectangle (3.9,0.9);
        \node at (3.5,.5) {\textbf{$4$}};        
        \node at (3.5,-.25) {\textbf{$4$}};      
        \node at (2,-1) {$\alpha=(3,1,1,2)$};
        \node at (2,-1.5) {$\Lucky(\alpha)=2$};
        
        \draw[black,very thick] (6.075,0.075) rectangle (6.925,0.925);
        \draw[black,very thick] (7.075,0.075) rectangle (7.925,0.925);
        \draw[black,very thick] (9.075,0.075) rectangle (9.925,0.925);
        \draw[step=1cm,gray,very thin] (6,0) grid (10,1);
        \draw[fill=skyblue] (6.1,0.1) rectangle (6.9,0.9);
        \node at (6.5,.5) {\textbf{$2$}};
        \draw[fill=skyblue] (7.1,0.1) rectangle (7.9,0.9);
        \node at (7.5,.5) {\textbf{$1$}};        
        \draw[fill=skyblue] (9.1,0.1) rectangle (9.9,0.9);
        \node at (9.5,.5) {\textbf{$3$}};        
        \draw[fill=vermillion] (9.1,1.1) rectangle (9.9,1.9);
        \node at (9.5,1.5) {\textbf{$4$}};
        \node at (6.5,-.25) {\textbf{$1$}};
        \node at (7.5,-.25) {\textbf{$2$}};
        \node at (8.5,-.25) {\textbf{$3$}};
        \node at (9.5,-.25) {\textbf{$4$}};
        \draw[->] (10,1.5) to (10.5,1.5) {};
        \node at (11,1.55) {\textbf{exits}};
        \node at (8,-1) {$\alpha=(2,1,4,4)$};
        \node at (8,-1.5) {$\Lucky(\alpha)=3$};
    \end{tikzpicture}
    \caption{
    Examples of parking preference lists, with cars that park highlighted in blue, lucky cars marked in black, and cars that fail to park highlighted in red.
    On the left, the preference list $\alpha = (3,1,1,2)$ is a parking function and has two lucky cars.
    On the right, the preference list $\alpha = (2,1,4,4)$ is not a parking function and has three lucky cars.
    }
    \label{fig:pf example}
\end{figure}

We study a special statistic of preference lists defined as follows.
If car $i$ parks in its preferred spot $p_i$, we say it is \emph{lucky}.
For a preference list $\alpha=(p_1,p_2,\ldots,p_n)\in[n]^n$ we consider the statistic 
\begin{align*}
    \Lucky(\alpha) \coloneqq |\{i\in[n]:\mbox{car $i$ is lucky}\}|,
\end{align*}
which simply counts the number of lucky cars in the preference list $\alpha$.
Note that we do not require $\alpha$ to be a parking function. For the parking preferences in Figure~\ref{fig:pf example}, note that  $\Lucky(3,1,1,2)=2$ and $\Lucky(2,1,4,4)=3$.
In fact, $\Lucky(\alpha) \geq 1 $ for all $\alpha \in [n]^n$, as the first car always parks in its preferred spot (and hence car 1 is always lucky).
If $S_n$ denotes the set of permutations of $[n]$, then $\Lucky(\alpha) = n$ for all $\alpha \in S_n$, as all cars have distinct preferences and park precisely in their preferred spot (and hence are all lucky).

Gessel and Seo studied the lucky statistic, and gave the following generating function~\cite[Theorem 10.1]{GesselSeo}:
\begin{align}
\sum_{\alpha\in\PF_{n}}
q^{\Lucky(\alpha)}
=q\prod_{i=1}^n(i+(n-i+1)q).\label{gessel seo poly}
\end{align}
In what follows, we let 
\[\mathcal{L}_n=\{\alpha\in[n]^n: \Lucky(\alpha)=n-1\}\]
and give a formula for $L_n=|\mathcal{L}_n|$; the number of preference lists in $[n]^n$ with $n-1$ lucky cars. 
Note that unlike \eqref{gessel seo poly}, where the sum is over parking functions of length $n$, $\mathcal{L}_n$ is not restricted to elements of $\PF_n$.

Our main result shows that $L_n = A_n$, where $A_n$ is given by \eqref{eq: A_n formula} and counts the \emph{total} number of comparisons performed by \texttt{Quicksort} given \emph{all} possible orderings of an array of size~$n$. 
\begin{theorem}
\label{thm:equal}
If $n\geq 2$, then
$
    L_n=A_n=n!\, Q_n.
$
\end{theorem}

\begin{proof}
It suffices to show that both $A_n$ and $L_n$ satisfy the second order recurrence relation
\begin{align}
    f_n=2nf_{n-1}-n(n-1)f_{n-2}+2(n-1)!
\hspace{.2in}\mbox{with $f_0=f_1=0$.}
\label{recurrence}        
\end{align}
Algebraic manipulations involving telescoping sums show that \eqref{eq: Q_n recursive} implies $n Q_n=(n+1)Q_{n-1}+2(n-1)$ 
with $Q_0=0$.
As $A_n=n!Q_n$ we have that $A_n =(n+1)A_{n-1}+2(n-1)(n-1)!$ with $A_0=0$.
Then,
\begin{align*}
    A_n 
    &=2nA_{n-1}-(n-1)A_{n-1}+2(n-1)(n-1)!\\
    &=2nA_{n-1}-(n-1)(nA_{n-2}+2(n-2)(n-2)!)+2(n-1)(n-1)!\\
    &=2nA_{n-1}-n(n-1)A_{n-2}+2(n-1)!
\end{align*}
with $A_0=A_1=0$, as desired.

We now show that $L_n$, which is the number of preference lists in $[n]^n$ with exactly $n-1$ lucky cars, also satisfies \eqref{recurrence}. 
Note that each $\alpha \in \mathcal{L}_n$ contains exactly one duplicate entry; all other entries are pairwise distinct. 
Hence, there is exactly one pair of cars that share a preference; call this pair of cars the \emph{competing cars}.
This also implies that there is exactly one parking spot that is not preferred by any car; call this parking spot the \emph{undesired spot}. 
Note that 
\begin{align}
    L_n&=|N_n|+|M_n|,\label{eq:rec 1}
\end{align}
where 
\begin{align*}
    N_n&=\{\alpha\in\mathcal{L}_n: \mbox{car 1 is not a competing car}\}, \mbox{ and }\\
    M_n&=\{\alpha\in\mathcal{L}_n: \mbox{car 1 is a competing car}\}.
\end{align*}

We begin by enumerating the elements of $N_n$. To do so we use the following fact. 
If car $1$ is not a competing car and prefers spot $a_1$, then the remaining $n-1$ cars behave as if car 1 preferred spot $n$.
Namely, by shifting the preferences so that $a_i = a_i - \mathbf{1}_{\{a_i > a_1\}}$ for each $i \in [n] \setminus \{1\}$, where $\mathbf{1}$ denotes the indicator function, we note there are $L_{n-1}$ 
preference lists that satisfy 
the required condition.
As there are $n$ options for the preference $a_1$ of car $1$ we have that
\begin{align}
    |N_n|&=nL_{n-1}\label{N_n}.
\end{align}

Now, if $\alpha=(a_1,a_2,\ldots,a_n)\in M_n$, then car $1$ is a competing car.
This implies that there exists an index $2\leq j\leq n$ such that $a_1=a_j$, while all other entries are pairwise distinct as well as distinct from this value. 
We now consider the cases $j=2$ and $3\leq j\leq n$ separately:
\begin{itemize}
    \item 
    If $j=2$, then $a_1=a_2$. 
    If $a_1=a_2=n$, then cars $3,4,\ldots,n$ have pairwise distinct preferences among the first $n-1$ parking spots. 
    Hence, if $a_1=a_2=n$, there are $(n-1)!$ possibilities for the remaining entries $a_3,a_4,\ldots,a_n$ in $\alpha$. 
    If $a_1=a_2=k<n$, then spot $k+1$ is the undesired spot and cars $3,4,\ldots,n$ have pairwise distinct preferences among the $n-1$ parking spots $\{1,2,\ldots,k-1,k+2,\ldots,n\}$.
    Hence, if $a_1=a_2=k<n$, there are $(n-1)!$ possibilities for the remaining entries $a_3,a_4,\ldots,a_n$ in $\alpha$.
    Together, these mutually exclusive cases contribute $2(n-1)!$ toward the total count.
\item 
    If $3\leq j\leq n$, then car $2$ is not a competing car. 
    Note that there are $n$ possible preferred spots for car $2$. 
    Moreover, the preference list $(a_1,a_3,a_4,\ldots,a_n)$ contains the two competing cars (with the first car being in the competing pair), and all of the remaining $n-1$ cars have preferences in the set $[n]\setminus\{a_2\}$. 
    By shifting the preferences so that $a_i = a_i - \mathbf{1}_{\{a_i > a_2\}}$ for each $i \in [n] \setminus \{2\}$, we note that the total number of such preferences for cars $1,3,4,\ldots,n$ is given by $|M_{n-1}|$.
    This case contributes $n|M_{n-1}|$ toward the total count.
\end{itemize}
Therefore,
\begin{align}
|M_n|&=n|M_{n-1}|+2(n-1)!.\label{M_n}
\end{align}
Substituting \eqref{N_n} and \eqref{M_n} into \eqref{eq:rec 1} yields
\begin{align}
    L_n=nL_{n-1}+n|M_{n-1}|+2(n-1)!.\label{eq:recurrence for Ln}
\end{align}
Subtracting $2nL_{n-1}$ from both sides of \eqref{eq:recurrence for Ln} yields
\begin{align*}
    L_n-2nL_{n-1}&=-nL_{n-1}+n|M_{n-1}|+2(n-1)!\\
    &=-n((n-1)L_{n-2}+|M_{n-1}|)+n|M_{n-1}|+2(n-1)!\\
    &=-n(n-1)L_{n-2}+2(n-1)!
\end{align*}
which upon rearranging is the desired recurrence relation
\begin{align*}
    L_n=2nL_{n-1}-n(n-1)L_{n-2}+2(n-1)!
\end{align*}
with $L_0=L_1=0$. 
\end{proof}
\begin{remark}
In Section~\ref{sec: alternative} we give an alternative proof of Theorem \ref{thm:equal}; we establish a rather technical binomial identity which shows $L_n=A_n$ for $n \geq 2$.
However, given the proofs of Theorem \ref{thm:equal}, it is natural to ask for a more fine-grained relationship between lucky cars and pivot comparisons in the \texttt{Quicksort} executions. 
Given the many algebraic manipulations used to arrive at either proof of Theorem \ref{thm:equal}, it is unclear how these two (or objects derived from them) are related. 
We welcome an analysis of this relationship.
\end{remark}

\section{Alternative Proof of Main Result}
\label{sec: alternative}

In this section we present an alternative proof of our main result. 
This approach relies on the following combinatorial enumeration.
\begin{theorem}\label{thm:lucky count}
If $n\geq 2$, then the number of parking preference lists $[n]^n$ with exactly $n-1$ lucky cars is given by
\begin{align}
L_n&=2\sum_{i=2}^n\sum_{\ell=0 }^{i-1}\ell(n-i+1)\frac{(i-1)!(n-\ell-1)!}{(i-1-\ell)!}.\label{our formula count}
\end{align}
\end{theorem}
\begin{proof}
Assume that car $i$ is the unlucky car, so $1<i\leq n$. 
Note that either car $i$ is unlucky and is unable to park, or car $i$ is unlucky and is able to park.

In the first case, there must be a subset of the cars $1$ through $i-1$ that park contiguously at the end of the street. 
Let $\ell$ denote the number of such cars and note $1\leq \ell\leq i-1$.
Then the only way for car $i$ to be unlucky and unable to park is if car $i$ has preference in the spots $n-\ell+1,n-\ell+2,\ldots,n$, of which there are $\ell$ options. 
Now we must select $\ell$ cars among the cars $1$ through $i-1$ to park contiguously at the end of the street. 
We can do this in $\binom{i-1}{\ell}$ ways, and we can permute their parking order among those spots in $\ell!$ ways so that they remain lucky. 
Now the remaining $i-1-\ell$ non-selected cars must park in some of the spots numbered $1$ through $n-\ell-1$. 
Note they cannot park in spot $n-\ell$ as that would mean that $\ell+1$ cars are parked contiguously at the end of the street, and we know we only have $\ell$ such cars. 
Then, there are $\binom{n-\ell-1}{i-1-\ell}$ ways to select their preferences, and we can permute their parking order among those spots in $(i-1-\ell)!$ ways so that they remain lucky. 
At this point we have accounted for cars $1$ through $i-1$ being lucky and car $i$ being unlucky and unable to park.
It remains for us to consider the preference for cars $i+1$ through $n$, all of which must also be lucky and can therefore only have preference among the remaining available spots on the street. 
Since car $i$ is unable to park and $i-1$ cars are parked on the street, we have $n-(i-1)$ unoccupied spots remaining. 
Then the remaining cars, of which there are $n-i$, can select among those open spots in $\binom{n-i+1}{n-i}$ ways, and we can permute their parking order among those spots in $(n-i)!$ ways so that they remain lucky. 
Thus the total count of parking preferences ensuring that car $i$ is unlucky and unable to park, while all of the other cars are lucky is given by
\begin{align}
& \sum_{i=2}^n\sum_{\ell=1}^{i-1}\ell\binom{i-1}{\ell}\ell!\binom{n-\ell-1}{i-1-\ell}(i-1-\ell)!\binom{n-i+1}{n-i}(n-i)!.\label{eq:count1}
\end{align}

In the second case, there can still be a subset of the cars $1$ through $i-1$ which park contiguously at the end of the street. 
Let $\ell$ denote the number of such cars and note $0\leq \ell\leq i-1$. 
Then the only way for car $i$ to be unlucky and park on the street is if car $i$ has preference in the spots not occupied by the cars parked in the last $\ell$ spots at the end of the street.
Now we must select $\ell$ cars among the cars $1$ through $i-1$ to park contiguously at the end of the street. 
We can do this in $\binom{i-1}{\ell}$ ways, and we can permute their parking order among those spots in $\ell!$ ways so that they remain lucky. 
Now the remaining $i-1-\ell$ non-selected cars must park in some of the spots numbered $1$ through $n-\ell-1$. 
Note they cannot park in spot $n-\ell$ as that would mean that $\ell+1$ cars are parked contiguously at the end of the street, and we know we only have $\ell$ such cars. 
Then, there are $\binom{n-\ell-1}{i-1-\ell}$ ways to select their preferences, and then we can permute those in $(i-1-\ell)!$ ways. 
At this point we have accounted for cars $1$ through $i-1$ being lucky with $\ell$ of those cars parking at the end of the street. 
Now for car $i$ to be unlucky and able to park, its parking preference must be one of the $i-1-\ell$ spots occupied by the $i-1-\ell$ cars not parked at the end of the street.

It remains for us to consider the preference for cars $i+1$ through $n$, all of which must also be lucky and can therefore only have preference among the remaining available spots on the street. 
Since cars $1$ through $i$ parked on the street, we have $n-i$  unoccupied spots remaining.
Then the remaining $n-i$ must choose among those spots and they can do so in $(n-i)!$ ways.
Thus the total count of parking preferences ensuring that car $i$ is unlucky and able to park, while all of the other cars are lucky is given by
\begin{align}
\sum_{i=2}^n\sum_{\ell=0}^{i-1}(i-1-\ell)\binom{i-1}{\ell}\ell!\binom{n-\ell-1}{i-1-\ell}(i-1-\ell)!(n-i)!.\label{eq:count2}
\end{align}

The result follows from the fact the two cases considered are disjoint. 
We remark that through straightforward algebraic manipulations involving binomials one can verify that in fact the counts are equinumerous, namely that the expressions in \eqref{eq:count1} and \eqref{eq:count2} are equal, and that their sum simplifies to the following expression:
\[L_n=2\sum_{i=2}^n\sum_{\ell=0 }^{i-1}\ell(n-i+1)\frac{(i-1)!(n-\ell-1)!}{(i-1-\ell)!}.\qedhere\]
\end{proof}
The proof of Theorem \ref{thm:lucky count} immediately implies the following.
\begin{corollary}
If $n\geq 2$, then
\begin{align*}
    \frac{L_n}{2} 
    &= |\{\alpha \in [n]^n \setminus \PF_n : \Lucky(\alpha) = n-1\}| \\
    &= |\{\alpha \in \PF_n : \Lucky(\alpha) = n-1|\}.
\end{align*}
\end{corollary}

We now complete our alternative proof.
\begin{proof}[Proof of Theorem \ref{thm:equal}]
By Theorem~\ref{thm:lucky count}, it suffices to show that the following holds for $n\geq 2$:
\begin{align*}
    2\sum_{i=2}^n\sum_{\ell=0 }^{i-1}\ell(n-i+1)\frac{(i-1)!(n-\ell-1)!}{(i-1-\ell)!}=n![2(n + 1) H_n - 4n].
\end{align*}
By induction on $i$ and $n$ with $i \leq n$, and using standard techniques manipulating binomials one can establish the following identity:
\begin{align}
    \sum_{\ell=0}^{i-1} \ell \binom{n-\ell-1}{i-\ell-1} 
    =\binom{n}{i - 2}.\label{identity}
\end{align}
Using identity \eqref{identity} in line \eqref{eq: using identity} below, we have
\begin{align}
    L_n 
    &=2\sum_{i=2}^n\sum_{\ell=0}^{i-1} \ell(n-i+1) \frac{(i-1)!(n-\ell-1)!}{(i-1-\ell)!} \nonumber \\
    &= 2\sum_{i=2}^n (n-i+1) (i-1)! (n - i)! \sum_{\ell=0}^{i-1} \ell \binom{n-\ell-1}{i-\ell-1} \nonumber \\
    &= 2\sum_{i=2}^n (n-i+1) (i-1)! (n - i)! \binom{n}{i - 2} \label{eq: using identity} \\
    &= 2\sum_{i=2}^n (i-1) \frac{n!}{(n - i + 2)} \nonumber \\
    &= 2n!\sum_{i=1}^{n-1} \frac{i}{(n + 1) - i} \nonumber \\
    &= 2n!\left[\sum_{i=1}^n \left(\frac{i}{(n + 1) - i}\right) - n\right] \nonumber \\
    &= 2n!\left[\left(\frac{1}{n} + \frac{2}{n-1} + \cdots + \frac{n}{1} \right)- n\right] \nonumber \\
    &= 2n!\left[\left(H_n + H_{n-1} + \cdots + H_1 \right)- n\right] \nonumber \\
    &= 2n!\left[\left((n+1)H_n - n\right) - n\right]\nonumber \\
    &= n! \left[2(n + 1) H_n - 4n \right] \nonumber \\
    &=n!Q_n \nonumber \\
    &=A_n, \nonumber
\end{align}
as desired.
\end{proof}

\bibliographystyle{plain}
\bibliography{bibliography}

\end{document}